\documentclass[12pt]{amsart}
\usepackage{amsmath,amsfonts,amssymb,stmaryrd}
\usepackage{graphicx,psfrag,subfigure}
\usepackage{amssymb,latexsym}
\usepackage{enumerate}

\makeatletter
\@namedef{subjclassname@2010}{%
  \textup{2010} Mathematics Subject Classification}
\makeatother

\newtheorem{thm}{Theorem}
\newtheorem{rk}{Remark}
\newtheorem{prop}{Proposition}

\newtheorem{lemma}{Lemma}
\newtheorem{defi}{Definition}

\newcommand{\R}{{\mathbb{R}}}

\newcommand{\Z}{{\mathbb{Z}}}

\begin{document}

\title{$C^1$ stability of endomorphisms on two dimensional manifolds.}

\author[J. Iglesias]{J. Iglesias}
\address{Universidad de La Rep\'ublica. Facultad de Ingenieria. IMERL.
Julio Herrera y Reissig 565. C.P. 11300. Montevideo, Uruguay}
\email{jorgei@fing.edu.uy }
\author[A. Portela]{A. Portela}
\address{Universidad de La Rep\'ublica. Facultad de Ingenieria. IMERL.
Julio Herrera y Reissig 565. C.P. 11300. Montevideo, Uruguay }
\email{aldo@fing.edu.uy }
\author[A. Rovella]{A. Rovella}
\address{Universidad de La Rep\'ublica. Facultad de Ciencias.
Centro de Matem\'atica. Igu\'a 4225. C.P. 11400. Montevideo,
Uruguay} \email{leva@cmat.edu.uy}

\date{\today}
\begin{abstract}
A set of necessary conditions for $C^1$ stability of noninvertible
maps is presented. It is proved that the conditions are
sufficient for $C^1$ stability in compact oriented manifolds of
dimension two. An example given by F.Przytycki in 1977 is
shown to satisfy these conditions. It is the first example known of a
$C^1$ stable map (noninvertible and nonexpanding) in a manifold of
dimension two, while a wide class of examples are already known in every other dimension.
\end{abstract}


\maketitle
\section{Introduction}
After the work of many authors, it was proved by C.Robinson
(\cite{rob}) and R.Ma\~{n}\'e (\cite{ma}) that a
diffeomorphism of a compact manifold is $C^1$ structurally stable
if and only if it satisfies the Axiom A and the strong transversality
condition.
It is also known since M.Shub \cite{sh} that an expanding map is
also $C^1$ stable.
On the other hand, Anosov endomorphisms not falling in the above
categories, fail to be stable (\cite{prz1} and \cite{mp}).\\

The following is the definition of (strong) Axiom A:
\begin{defi}
\label{def1}
A $C^1$ map $f$ satisfies the Axiom A if:\\
(A1) the nonwandering set $\Omega_f$ has a hyperbolic structure,\\
(A2) the set of periodic points is dense in $\Omega_f$, and\\
(A3) each basic piece is either expanding or the restriction of $f$ to it is injective.\\
\end{defi}

In a compact manifold $M$, the following properties are necessary for a map $f\in C^1(M)$ to be $C^1$ structurally stable:\\
{\bf (C1)} The set of critical points of $f$ is empty (a point $x$ is critical or singular for $f$ if the differential
of $f$ at $x$ is noninvertible).\\
{\bf (C2)} The map $f$ is Axiom A without cycles.\\
{\bf (C3)} If the unstable set of a basic piece $\Lambda$ intersects another basic set $\Lambda '$, $W^u(\Lambda )\cap\Lambda^{'}\neq\emptyset$, then $\Lambda$ is an expanding basic piece.

Condition (C1) is obviously necessary for $C^1$ structural
stability. Examples of $C^r$ ($r\geq 2$)
stable maps with critical points can be easily found for maps in
manifolds of dimension one;
some examples in dimension two were recently discovered (see \cite{ipr1}).
A theorem by P.Berger \cite{ber}, gives sufficient
conditions for $C^\infty$ stability of maps having critical points.\\

F.Przytycki showed (\cite{prz}) that if $f$ is a $C^1$-$\Omega$ stable
map that satisfies Conditions (A1) and (A2),
then it also satisfies (A3). It follows that an Anosov
endomorphism is not $\Omega$ stable unless it is expanding or a diffeomorphism. Later, adapting the
proof of the $C^1$ stability conjecture given by R.Ma\~{n}\'e, it was proved by
N.Aoki, K.Moriyasu and N.Sumi that if $f$ is a $C^1$-$\Omega$ stable map without critical
points then $\Omega(f)$ has a hyperbolic structure (\cite{ams}). These two results imply
that (C2) is a necessary condition for $C^1$ stability.\\
That item (C3) is a necessary condition for $C^1$ stability was also shown by F.Przytycki in Theorem C of his above mentioned article.\\
Still in that article, F.Przytycki gave an example of a map in the two
torus satisfying the three conditions above, and asked if the map
obtained is structurally stable. The first objective of our work was
to prove the stability of this map; we finally arrive to a
characterization of the $C^1$ stability in dimension
two. Unfortunately, the techniques applied do not allow a generalization
to higher dimensions.  The same characterization of the $C^1$
stability in higher dimensions stands as a conjecture.\\
Note that a self map of a compact manifold must be a covering map if
it is singular points free. It follows that the unique two dimensional
orientable manifold admitting a noninvertible $C^1$ stable map is the
two torus. There exist examples of $C^1$ stable maps in
compact manifolds of dimension at least three.
However, that construction cannot be carried on in manifolds of
dimension two, because of the nature of the attractors (see \cite{ipr2}).
Indeed, the problem in dimension two is that it is not possible to
find examples of $C^1$ stable maps without saddle type basic pieces,
and the difficulty of having that type of basic pieces is that the
unstable manifolds could have self intersections.\\
If two (or more) unstable manifolds intersect at a point $z$, the
{\em Strong Transversality Condition} requires that they intersect
transversally and that the intersection is transverse to the stable
manifold through $z$: see definition 3 in the next section. Here we prove:

\begin{thm}
\label{unico}
The following conditions are necessary for an endomorphism $f$ of a compact manifold $M$ to be $C^1$ stable:
\begin{enumerate}
\item
$f$ has no critical points.
\item
$f$ satisfies the Axiom A.
\item
$f$ satisfies the Strong Transversality Condition.
\end{enumerate}
If $M$ is a two dimensional manifold, then the above conditions are also sufficient for $C^1$ stability of the map $f$.
\end{thm}

\section{Definitions}
Since the hypothesis of stability implies that the map satisfies the
Axiom A and has no critical points, these facts can be
assumed throughout the whole article.
{\bf Notation:} Let $\Omega$ denote the nonwandering set of $f$,
$A=A(f)$ the union of the attracting basic pieces, $E=E(f)$
the union of the expanding basic pieces and $\Gamma=\Gamma (f)=\Omega\setminus (A\cup E)$.\\
The Axiom A condition implies
that the restriction of $f$ to $\Gamma\cup A$ is injective. The union
of the attracting periodic orbits will be denoted by $A_{per}$.
Recall that a basic piece $\Lambda$ is called expanding if there exist
constants $c>0$ and $\lambda>1$ such that
$|Df^n(v)|\geq c\lambda^n|v|$ for every $n\geq 0$ and $v\in
T_\Lambda(M)$. It follows that a basic piece $\Lambda$ is contained
in $\Gamma$ if and only if it has an unstable manifold of dimension less than the
dimension of $M$ and this unstable manifold is not
contained in $\Lambda$.\\
Definitions and basic properties of invariant manifolds associated to
basic pieces of Axiom A maps are presented in the article of F.Przytycki
\cite{prz1}.
We will recall here some of these results.\\
{\em Stable manifolds. }
The stable set of a point $x\in\Omega$ is denoted by $W^s(x)$ and defined as the set of points
$y\in M$ such that
$d(f^n(x),f^n(y))\to 0$ as $n\to +\infty$, where $d$ is the distance
induced by the Riemannian metric on $M$.
For each point $x\in \Gamma\cup A$ and $\epsilon>0$ sufficiently
small, the local stable manifold of $x$ is defined as
the set $W^s_\epsilon(x)$ of points $y\in M$ such that
$d(f^n(x),f^n(y))<\epsilon$ for every $n\geq 0$. It is known that
$W^s_\epsilon(x)$ is tangent to the stable space at $x$ and that
$\{W^s_\epsilon(x)\ :\ x\in\Gamma\cup A\}$ is a family of $C^1$
embedded disks. The stable set of a basic piece $\Lambda$ is defined as the union of the stable sets of points in $\Lambda$.
The path connected component of $W^s(x)$ containing the point $x$ will be denoted by $W^s_0(x)$.\\
{\em Unstable manifolds. }
If $x\in \Gamma\cup A$, then there exists a unique preorbit $\{x_n\}$
of $x$ contained in $\Gamma\cup A$. Provided $\epsilon>0$ is
sufficiently small, it is adequate to define the local unstable
manifold of $x$ as the set $W^u_\epsilon(x)$ of points $y$ in $M$ having a preorbit $\{y_n\}$ such that $d(x_n,y_n)<\epsilon$ for every $n\geq 0$. If $\Lambda\subset\Gamma$ is a basic piece then there is a well defined unstable space $E^u(x)$ for each $x\in\Lambda$ that is invariant and expanded by the differential. It is known that $W^u_\epsilon(x)$ is tangent to the unstable space at $x$ and that $\{W^u_\epsilon(x)\ :\ x\in\Gamma\cup A\}$ is a continuous family of $C^1$ embedded disks. The unstable set of $x\in \Gamma\cup A$ is defined as $\bigcup_{n>0}f^n(W^u_\epsilon(x_n))$. An equivalent definition would be the following: $y\in W^u(x)$ if and only if there exists a preorbit $\{y_n\}$ of $y$ such that $d(x_n,y_n)\to 0$, where $\{x_n\}$ is the preorbit of $x$ contained in $\Gamma$. Finally define $W^u(\Lambda)$ as the union of the unstable sets of points in $\Lambda$, and $W^u(\Gamma)$ as the union of the $W^u(\Lambda)$ for $\Lambda\subset\Gamma$. If $\Lambda$ is an expanding basic piece, then there exists a neighborhood $U$ of $\Lambda$ such that the closure of $U$ is contained in $f(U)$ and the intersection of the backward iterates of $U$ is equal to $\Lambda$. The unstable set $W^u(\Lambda)$ of a basic piece $\Lambda\subset E$ is defined as the union of the future iterates of $U$.\\

\noindent
{\em Properties of stable and unstable sets.}\\
\begin{enumerate}
\item
The path component $W^s_0(x)$ of $W^s(x)$ containing $x$ is an injective immersion of an Euclidean space $\R^n$. Moreover, $W^s_0(x)$ can be obtained as the union, for $n\geq 0$, of the connected component of $f^{-n}(W^s_\epsilon(f^n(x)))$ that contains $x$. For every $x\in \Gamma\cup A$ it holds that:
$$
W^s(x)=\bigcup_{n\geq 0}f^{-n}(W^s_0(f^n(x))).
$$
Therefore, each component of the stable set of a point $x$ is always an injective immersed manifold.
Moreover, different stable sets cannot intersect.
\item
For each point $z\in M$ there exists a point $x\in \Omega$ such that $z\in W^s(x)$. As stable sets are submanifolds, the stable space $E^s(z)=T_z(W^s(x))$ is well defined for every $z\in M$.
\item
Unstable sets of different points can have nonempty intersection. Moreover, the unstable set of a point $x\in \Omega$ is a (not necessarily injective) immersed manifold.
\item
The stable set of a basic piece $\Lambda$, defined as $W^s(\Lambda)=\cup_{x\in \Lambda} W^s(x)$, is backward invariant ($f^{-1}(W^s(\Lambda))=W^s(\Lambda)$).
The unstable set of a basic piece $\Lambda$ is forward invariant but not necessarily backward invariant.
\end{enumerate}

\begin{defi}
\label{order}
Let $\Lambda_1$ and $\Lambda_2$ be different basic pieces of an Axiom A map $f$. Say that $\Lambda_1>\Lambda_2$ if $W^u(\Lambda_1)\cap W^s(\Lambda_2)\neq \emptyset$.
\end{defi}
If $f$ is an Axiom A map that satisfies the transversality condition, then
the relation above defines a partial order. Moreover, there exists an
adapted filtration for this order (see {\cite[proposition 1.1]{prz}}).\\
A closed set $L$ is attracting if there exists an open neighborhood
$N$ of $L$ such that the closure
of $f(N)$ is contained in $N$, and the intersection of the forward iterates of $N$ is equal to $L$.\\
The existence of filtrations for Axiom A map with no cycles implies the following result:
\begin{lemma}
\label{unstable}
If $f$ is an Axiom A map with no cycles, then $W^u(\Gamma)\cup A$ is a closed attracting set.
\end{lemma}

For each point $z\in W^u(\Gamma)$ there exists a preorbit $\{z_n\}$ of $z$ satisfying $d(x_n,z_n)\to 0$, where $\{x_n\}$ denotes
the unique preorbit of $x$ contained in $\Gamma$. It follows that
there exists $k_0$ such that $z_k\in W^u_\epsilon(x_k)$
for every $k\geq k_0$, and therefore one has a well defined space
$$
E^u(\{z_n\})= Df_{z_k}^{k}(T_{z_k}(W^u_\epsilon(x_k)))\subset T_z(M).
$$
But there may exist other preorbits of $z$ converging to $\Gamma$. Let
$\alpha_z$ be the collection of subspaces
$E^u(\{z_n\})\subset T_zM$ indexed by the different preorbits $\{z_n\}$ of $z$ in $W^u(\Gamma)$.\\
Given a finite dimensional vector space $X$, say that a collection of subspaces $\{X_i\}$ of $X$ is in general position if the sum of the codimensions of the $X_i$ is equal to the codimension of $\cap_i X_i$.
\begin{defi}
\label{stc}
An Axiom A map satisfies the Strong Transversality condition if
\begin{enumerate}
\item
The collection $\alpha_z$ is in general position.
If this is the case, the intersection of the $E^u(\{z_n\})$ is denoted by $E^u(z)$,
and we have a well defined $E^u(z)$ for every $z\in M$.
\item
If $z\in W^u(\Gamma)$, then $E^s(z)+E^u(z)=T_z(M)$.
\end{enumerate}
\end{defi}

This definition is intended to assure not only the transversality of intersections between stable and unstable manifolds, but
also between different unstable manifolds. This definition is more restrictive, even if compared with the definition of transversality given in   {\cite[definition 2.3]{br}}that was intended to characterize inverse stability.

When a point $z$ has at least two different preorbits in the unstable set of $\Gamma$, then $z$ is called an unstable intersection. This can also be defined as follows:
\begin{defi}
\label{I}
A point $z\in M$ is an unstable intersection of $f$ if there exist some positive integer $k$ such that $f^{-k}(z)$ contains at least two points in $W^u(\Gamma)$. The set of unstable intersections is denoted by $I=I(f)$. Define also the first unstable intersections of $f$ as the set $I_1=I_1(f)$ of points $z$ for which the above property holds with $k=1$.
\end{defi}
Note that $I$ is a forward but not backward invariant set.\\
\begin{rk}
The Strong Transversality condition implies Przytycki's necessary condition of stability (see (C3) in the introduction).\\
Proof:
Assuming that the Strong Transversality condition holds, that $\Lambda$ is a basic piece and that there exists a point $x\in \Omega\setminus(E\cup\Lambda)$ such that $W^u(x)\cap \Lambda\neq\emptyset$, we must find a contradiction.
As $x\notin E$, then the codimension of $W^u(x)$ is positive.
If $y\in W^u(x)\cap \Lambda$, then $y\in I(f)$, and the collection $\alpha_y$ contains, at least, two elements: one (denoted $\{\bar y\}$) is the preorbit of $y$ contained in $\Lambda$, and another (denoted $\{\bar y_0\}$) converges to the basic piece that contains $x$. On one hand, note that $E^u(\bar y)$ and $E^s(y)$ are complementary subspaces because $\Lambda$ is a hyperbolic basic piece. On the other hand, $E^u(\bar y_0)$ has positive codimension because the basic set that contains $x$ is not expanding. It follows that the sum of the codimensions of $E^u(\bar y)$, $E^u(\bar y_0)$ and $E^s(y)$ is greater than the dimension of $M$. This implies that $f$ cannot satisfy both items of definition \ref{stc}.
\end{rk}

\section{Necessary conditions}

Begin with a $C^1$ structurally stable map $f$. It is clear that $f$ cannot have critical points. It is already known that $f$ is an Axiom A map. It remains to prove that $f$ satisfies the Strong Transversality Condition.
Much of our proof rests on the techniques employed by J.Franks in \cite{F} and in Kupka-Smale Theorem.\\
Let $\mathcal U$ be a neighborhood of $f$ where every map is conjugated to $f$. For each $g\in\mathcal U$, $x\in \Gamma(g)$ and a positive integer $R$, define
$$
W^u_R(x,g)= \bigcup_{k=0}^{k=R}g^k(W^u_\epsilon(x_k,g)),
$$
where $\{x_k\}$ is the unique $g$-preorbit of $x$ contained in $\Gamma(g)$. Define also $W^s_R(x,g)$ as the set of points at distance less than or equal to $R$ in $W^s(x,g)$, where the distance is induced by the Riemannian metric of $M$ in $W^s$: the distance between two points is the infimum of the length of curves joining the points within $W^s$.\\
Note that $\Omega\cap I=\emptyset$, ($I=I(f)$ of Definition \ref{I}): if $x\in I$ then $f^{-k}(x)$ has at least two points in $W^u(\Gamma)$, one of which cannot be contained in $\Omega$, but this contradicts Theorem C of \cite{prz} if $x\in \Omega$.\\

\noindent
{\em Proof of the first part of the Strong Transversality Condition. }\\

\begin{defi}
\label{dK1}
Fix positive integers $n$ and $R$. Given $g\in\mathcal U$ and $z\in W^u(\Gamma(g))\cap I(g)$ denote by $P_z(n,R,g)$ the set of points $w\in W^u(\Gamma(g))$ such that the following conditions hold:
\begin{enumerate}
\item
$w\in W^u_R(q_w,g)$ for some periodic point $q_w\in\Gamma(g)$ of period at most $n$.
\item
$w\in g^{-1}(I_1(g))\setminus I(g)$, and $g^k(w)=z$ for some positive $k=k_w$.
\end{enumerate}
Define $\mathcal K_1(n,R)$ as the set of maps $g\in\mathcal U$ such that the collection of subspaces
$$
\{Dg^{k_w}_w(T_w(W^u(w)))\ :\ w\in P_z(n,R,g)\}
$$
is in general position.
\end{defi}
Note that $T_w(W^u(w))$ is well defined since $w\notin I(g)$.

Using the same techniques applied to prove Kupka-Smale Theorem in the case of diffeomorphisms, one can conclude that for every $n$ and $R$ positive, $\mathcal K_1(n,R)$ is open and dense in $\mathcal U$.

It follows that the set $\mathcal K_1$, intersection of the sets $\mathcal K_1(n,R)$ for $n$ and $R$ positive integers is a residual set in $\mathcal U$.\\

Let $f$ be a $C^1$ structurally stable map, and assume that the first condition of definition \ref{stc} does not hold. Then there exist a point $z\in I$ such that the collection $\alpha_z$ is not in general position. This means that a finite subset $E^u(\{z^1_n\})\cdots , E^u(\{z^r_n\})$ is not in general position. As the preimages of $z$ are wandering and the periodic points are dense in $\Omega$, there exists a perturbation $g$ of $f$ such that for each $1\leq i\leq r$, the sequence $\{z^i_n\}$ belongs to the unstable set of a periodic point $p_i$ of $g$. As $g$ is conjugated to a map in $\mathcal K_1$, the sum of the codimensions of the subspaces $E^u(\{z^i_n\})$ is less than or equal than the dimension of $M$.
In addition, as the subspaces $E^u(\{z^i_n\})$ are not in general position, the arguments of Franks, {\cite[Lemma 2.1]{F}} imply that there exists a perturbation $g_1$ of $g$ such that the unstable manifolds of the periodic points $p_i$ of $g_1$ intersect in a submanifold of codimension less than the sum of their codimensions, but this contradicts the fact that $g_1$ must be conjugated to a map in $\mathcal K_1$.\\

\noindent
{\em Proof of the second part of the Strong Transversality Condition. }\\

\begin{defi}
\label{dK2}
Given positive integers $n$ and $R$, let $\mathcal K_2(n,R)$ be the set of maps $g\in\mathcal K_1(n,R)\cap \mathcal U$ such that the intersection of the subspaces $Dg^{k_w}_w(T_w(W^u(w)))$ for $w\in P_z(n,R,g)$ is transverse to $E^s(z)$ whenever $z$ belongs to $W^s_R(p,g)$ and $p$ is a periodic point of $g$ whose period is at most $n$.
\end{defi}

\begin{lemma}
\label{lK2}
For every $n$ and $R$, $\mathcal K_2(n,R)$ is open and dense in $\mathcal U$.
\end{lemma}
\begin{proof}
It is clear that the property is open and that the collection of subspaces $Dg^{k_w}_w(T_w(W^u(w)))$ for $w\in P_z(n,R,g)$ is transverse if $g\in\mathcal K_1(n,R)$. A perturbation supported in a neighborhood of $z$ will produce no changes in unstable manifolds but will make the stable manifold through $z$ transverse to the collection of subspaces.
\end{proof}

It follows that the intersection $\mathcal K_2$ of the $\mathcal K_2(n,N)$ is a residual subset of $\mathcal U$.

To prove the second assertion of the Strong Transversality Condition, assume by contradiction that there are points $x\in \Omega(f)$ and $y\in W^s(x)\cap  W^u(\Gamma)$ such that $E^s(y)$ and $E^u(y)$ are not transverse. We can assume, as above, that $x$ is a periodic point of $f$, and also that every $w\in P_y$ belongs to the unstable manifold of a periodic point. If the sum of the dimensions of $E^s(x)$ and $E^u(y)$ is greater than or equal to the dimension of the ambient manifold, one can produce a perturbation $g$ such that, if $W^u(x)$ is defined as
$$
W^u(x)=\cap_{w\in P_y}Dg_w^{n_w}(E^u(w)
$$
then $W^u(x)\cap W^s(x)$ contains a disc of codimension less than the sum of the codimensions of $W^s(x)$ and $W^u(x)$, which constitutes an obstruction to the equivalence with a map of $\mathcal K_2$.
If the sum of the dimensions is less than the dimension of $M$, then a perturbation in $\mathcal K_2$ will not be equivalent to $f$.

\section{Example}

The product of two $C^1$ stable maps cannot be $C^1$ $\Omega$-stable, unless both maps are diffeomorphisms or both are expanding. Indeed, if a stable map is not a diffeomorphism, then it has a non injective expanding basic piece, and if a stable map is not expanding, then it has a nontrivial attractor. The product of an expanding basic piece times an attractor is a basic piece that is neither expanding nor injective.
The example of Przytycki (see the final section of \cite{prz}) is a $C^0$ perturbation of the product map $(s,t)=(f_1(s),f_2(t))$, where $f_1$ and $f_2$ are maps of the circle which graphs are shown in figure two.
We will consider the two-torus as the product $\{(s,t):\ s\in[-\pi,\pi],\ t\in[-\pi,\pi]\}$, where $-\pi$ and $\pi$ are identified.

The map $f_1$ is a diffeomorphism having an attractor at $s=0$ and a repeller at $s=\pi$. The map $f_2$ is a $C^0$ perturbation of $z\to z^2$ ({\em derived from expanding} map). The nonwandering set of $f_2$ is the union of an attracting fixed point at $t=0$ and an expanding Cantor set $K_1$. The set $K$ defined as $\{0\}\times K_1$ is a hyperbolic isolated transitive saddle type set, but the restriction of the map to this set is not injective.\\
Przytycki proposed to perturb the product map as follows:
\begin{equation*}
f(s,t)=(f_1(s)+\sin(t)\varphi(s), f_2(t)),
\end{equation*}
where $\varphi$ is a smooth function satisfying two more assumptions. For every $s$ it holds that $0\leq \varphi(s)\leq \varepsilon$. There exists a constant $s_0$ such that $\varphi(s)=\varepsilon$ for $|s|<s_0$, and $\varphi(s)=0$ in $|s-\pi|<s_0$. The absolute value of the derivative of $\varphi$ is also smaller than $\varepsilon$. At each $s$, $|\varphi'(s)|<|f_1'(s)|$.

Other conditions are assumed for the maps $f_1$ and $f_2$. Both $f_1$ and $f_2$ are odd functions of $[-\pi,\pi]$. In a neighborhood of $0$, the function $f_1$ is equal to $s\to \lambda s$, where $0<\lambda<1$ is very small. The map $f_2$ has fixed points at $0$, $\delta$ and $-\delta$. Moreover $f_2$ is taken in such a way that the preimage of $\delta$ is $-\pi+\delta/3$ (so the preimage of $-\delta$ is $\pi-\delta/3$). The derivative of $f_2$ in $|t|>\delta$ is bigger than two. The constant $\varepsilon$ is taken sufficiently small; and once $\varepsilon$ is chosen, the constant $\lambda$ is taken still smaller.

\begin{figure}[ht]
\centering
\psfrag{-1}{$-\pi$}\psfrag{a}{$-\delta$}\psfrag{01}{$0$}
\psfrag{b}{$\delta$}\psfrag{1}{$\pi$}
\psfrag{p}{${{graph \ of \ f_2}}$}
\psfrag{q}{${{graph \ of \ f_1}}$}

\psfrag{pi}{$\pi$}\psfrag{0}{$(0,0)$}
\psfrag{r}{$R$}
\psfrag{b1}{$B_1$}\psfrag{b0}{$B_{0}$}\psfrag{delta}{$\delta$}
\psfrag{-delta}{$-\delta$}
\psfrag{delta0}{$s_0$}
\psfrag{-delta0}{$-s_0$}
\psfrag{-pi}{$-\pi$}
\psfrag{fr}{$f(R)$}
\psfrag{z}{$\mbox{The picture shows}$.}
\psfrag{z2}{$U=[-s_0,s_0]\times[-\pi ,\pi] \mbox{ and } f(U)$.}

\psfrag{z1}{ \mbox{The Shadowed region is } $f(R)$.}

\subfigure[]{\includegraphics[scale=0.201]{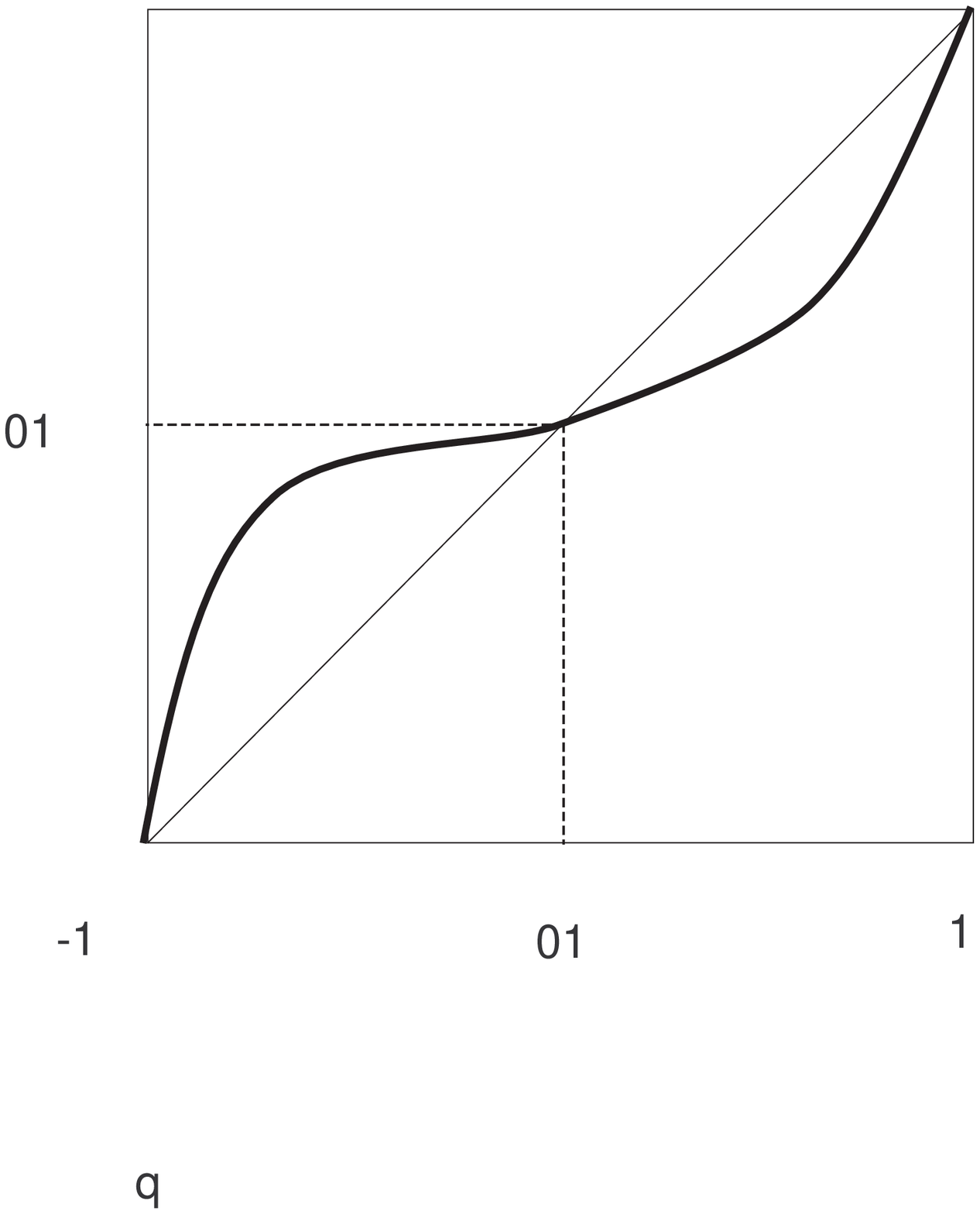}}
\subfigure[]{\includegraphics[scale=0.21]{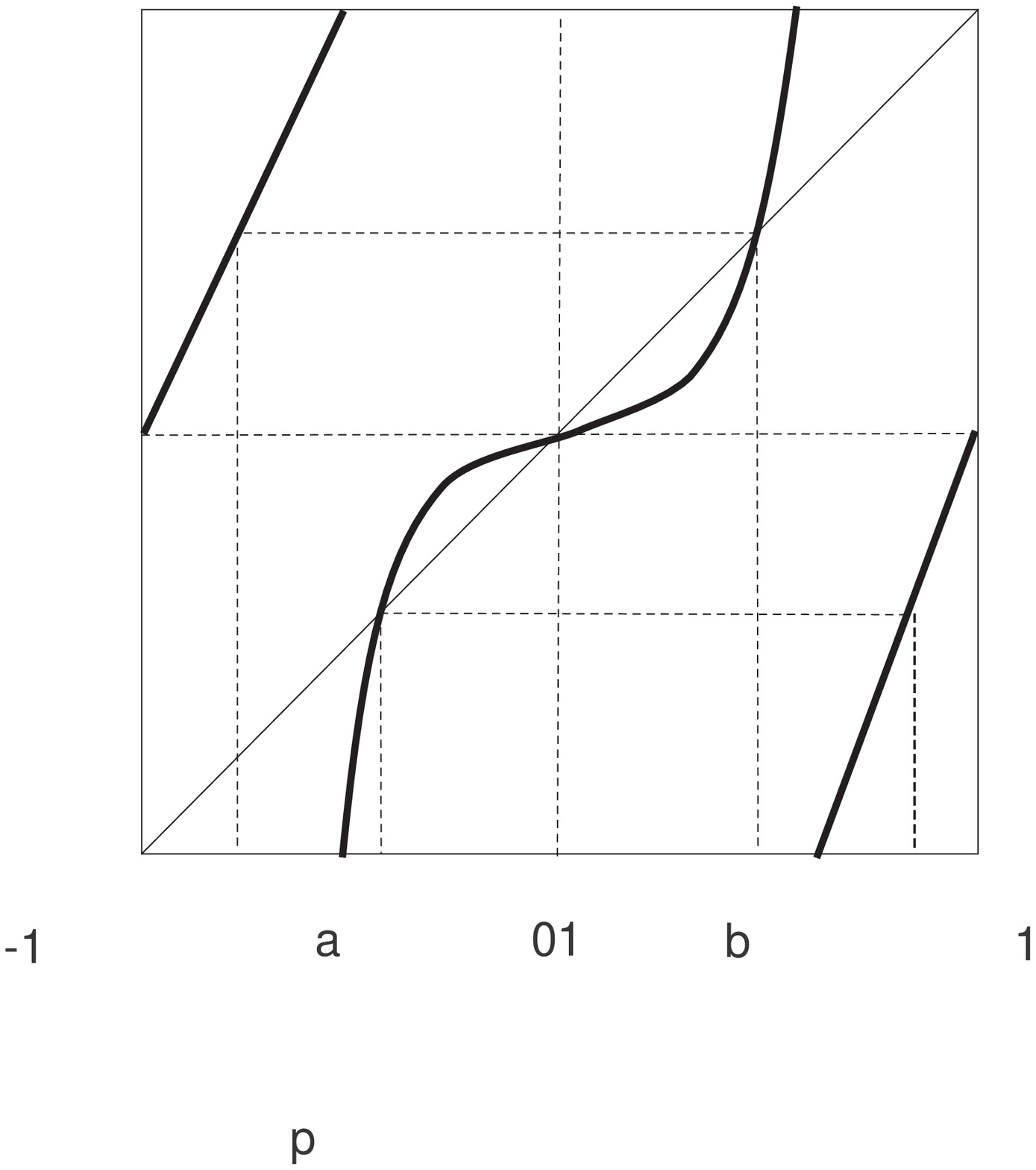}}
\caption{\label{sillaac}}
\end{figure}
\begin{figure}[ht]
\psfrag{pi}{$\pi$}\psfrag{0}{$(0,0)$}
\psfrag{r}{$R$}
\psfrag{b1}{$B_1$}\psfrag{b0}{$B_{0}$}\psfrag{delta}{$\delta$}
\psfrag{-delta}{$-\delta$}
\psfrag{delta0}{$s_0$}
\psfrag{-delta0}{$-s_0$}
\psfrag{-pi}{$-\pi$}
\psfrag{fr}{$f(R)$}

\begin{center}
\includegraphics[scale=.133]{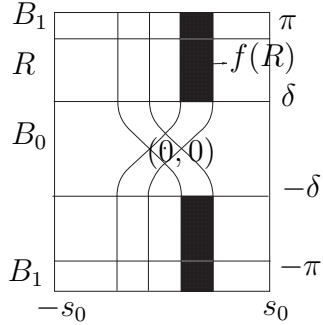}
\end{center}
\caption{\label{rot0.5} The picture shows $U=[-s_0,s_0]\times[-\pi ,\pi]$ and $f(U)$. The Shadowed region is $f(R)$.}
\end{figure}

\begin{prop}
\label{p1}
The map $f$ satisfies the necessary conditions for $C^1$ stability stated in Theorem \ref{unico}.
\end{prop}
\begin{proof}
{\bf (1) $f$ is a covering. } If $\varepsilon$ is small enough, then the derivative of the first coordinate with respect to $s$ never vanishes. This implies that $f$ is a local homeomorphism. As every point has exactly two preimages, the claim follows.\\
{\bf (2) $f$ satisfies the Axiom A. } The proof below will show that the nonwandering set of $f$ consists of the union of the following basic pieces:
\begin{enumerate}
\item
An attracting fixed point $p=(0,0)$.
\item
A saddle type fixed point $C=(\pi,0)$.
\item
An expanding set $K=\{\pi\}\times K_1$.
\item
A saddle type basic piece $\Gamma_0$ close to $\{0\}\times K_1$.
\end{enumerate}
That the first three are basic pieces is clear. Towards the proof of the existence, hyperbolicity and injectivity of $\Gamma_0$,
it will be necessary to construct invariant cone fields in the region $R=\{(s,t):|s|<s_0,\ |t|\in (\delta,\pi-\delta)\}$.
Note first that $f_1(\sigma)+\varphi(\sigma)\sin \delta=\sigma$ has the solution $\sigma=\epsilon\sin\delta/(1-\lambda)<s_0$. It follows that the points $A=(\sigma,\delta)$ and $-A=(-\sigma,-\delta)$ are fixed for $f$.

Assume that $v=(v_1,v_2)$ is a vector satisfying $|v_1/v_2|<\rho$. If $(w_1,w_2)=Df_X(v)$, where $X$ belong to the region $R$, then $|w_1/w_2|<\lambda\rho/2+\varepsilon/2$, where we have used that the derivative of $f_2$ is greater than $2$ and that $\varphi=\varepsilon$ in that region. If $\rho$ is taken greater than $\varepsilon/(2-\lambda)$ then the region $R$ admits an expanded almost vertical invariant cone field (because $s_0$ is greater than $\sigma+\rho\pi$ if $\varepsilon$ is sufficiently small). The same holds for the region $R'=\{(s,t):(s-t)\in R\}$.
The horizontal lines form a contracting invariant foliation in $R\cup R'$. Therefore the set $\Gamma_0$ defined as the set of points having a whole orbit contained in $R\cup R'$ has a hyperbolic structure; it follows also that $\Omega$ is hyperbolic and that $\Gamma_0$ is a basic piece. To obtain Axiom A it remains to prove the injectivity in $\Gamma_0$.
Note that the intersection of the future images of $R$ is a region bounded by segments of the unstable manifolds of the fixed points in $\Gamma_0$,  $A$ and $-A$. These are curves of the form $(\gamma(t),t)$ for $t\in [\delta,\pi-\delta]$, where $\gamma(t)$ is at a distance less than $\rho\pi$ from the first coordinate of $A$ or $-A$.
In any case it is easily seen that the image of a curve like that must be contained in $\{(s,t):s>0\}$, if $\lambda$ is taken sufficiently small. Analogously, the image of the region $R'$ is contained in $\{s<0\}$ by symmetry. It follows that the restriction of $f$ to $R\cup R'$ is injective.\\
{\bf (3) Strong Transversality Condition. } Obviously, stable and unstable manifolds intersect transversally. Moreover, the intersections of unstable manifolds occur in $B_0$, the immediate basin of $(0,0)$. It remains to show that these intersections are transverse. It is sufficient to prove that the intersections in $I_1$ (see section 2, definition 4) are transverse. A point $z$ belongs to $I_1$ if and only if one of the preimages of $z$ belongs to $B_0$ and the other belongs to $B_1$, where $B_1$ denotes the component of $f^{-1}(B_0)$ that is not $B_0$. The unstable manifolds of points in $\Gamma_0$ are almost vertical in $R\cup R'\cup B_1$. Note that if $\gamma(t)=(\alpha(t),t)$ is almost vertical, where $|t\pm \pi|<\delta$, then the image under $f$ of $\gamma$ is a curve contained in the immediate basin of $(0,0)$ and its tangent vector is $W_1=(\lambda\alpha'(t)+\varepsilon\cos t, f_2'(t))$. It can be parametrized as $\tilde\gamma(t')=(\beta(t'),t')$, with $t'$ now varying in $[-\delta,\delta]$. The image under $f$ of a curve of the form of $\tilde\gamma$ has tangent vector $W_2=(\lambda\beta'(t')+\varepsilon\cos t',f_2'(t'))$. As $t$ is close to $\pi$ and $t'$ close to $0$ the signal of the first coordinate of $W_1$ and $W_2$ are different if $\lambda$ is sufficiently small. But as the second coordinates are positive we conclude that these curves intersect transversally.
\end{proof}
%
%

\section{Sufficient conditions}
From now on, it is assumed that $f$ satisfies properties $(1)$ to $(3)$ of Theorem \ref{unico}.

We will first construct the conjugacy $h$ in fundamental domains of the periodic attractors.

\subsection{Controlling unstable intersections.}

To begin this section we translate the Strong Transversality Condition to simpler words in the case of dimension two.
Unstable manifolds of basic pieces in $\Gamma$ have dimension
one. This implies that for every positive $k$, and $z\in M$,
the intersection of $f^{-k}(z)$ with $W^u(\Gamma )$ contains at most two points. Then $E^u(z)$ (see definition \ref{stc}) has dimension zero whenever $z\in I$. Moreover, as unstable intersections must be transverse to stable manifolds, it follows that $I$ is contained in the basin of periodic attractors. The next result implies that $I_1$ (the set of first intersections, see definition \ref{I}) is compactly contained in the basin of $A_{per}$, the set of attracting periodic orbits.

\begin{lemma}
\label{closed}
If $f$ satisfies the Strong Transversality Condition, then $I_1$ is a closed set.
\end{lemma}
\begin{proof}
Let $\{z_n\}$ be a sequence of points in $I_1$ that converges to $z$, and for each $n$ let $z_n^1\neq z_n^2$ be points in $W^u(\Gamma)$ such that $f(z_n^1)=f(z_n^2)=z_n$. Assume that for $i=1,2$ the sequences $\{z_n^i\}$ converge to points $z^1$ and $z^2$, that must be different since $f$ is locally invertible. Then $z^1$ and $z^2$ belong to $W^u(\Gamma)\cup A$ as this is a closed set. But the strong transversality condition implies that neither $z^1$ nor $z^2$ can belong to $A$. Then $z=f(z^1)=f(z^2)$ belongs to $I$, and again by the strong transversality condition, $z\in I_1$.
\end{proof}

\begin{lemma}
\label{lemmaS}
Given a neighborhood $U$ of $I_1(f)$ and $\delta >0$, there exists a neighborhood $\mathcal U$ of $f$, and a positive integer $k$, such that, for every $g\in \mathcal U$, the following properties hold:
\begin{enumerate}
\item
The set of intersections $I_1(g)$ is contained in $U$ and $g^{-k}(I_1(g))\cap W^u(\Gamma(g))\subset W^u_\delta(\Gamma(g))$.\\
\item
There exists $\rho>0$ such that, if $x$ and $y$ are different points in $\Gamma(g)$, then the distance between two different points in $g^k(W^u_\delta(x,g))\cap g^k(W^u_\delta(y,g))$ is greater than $\rho$.
\end{enumerate}
\end{lemma}
\begin{proof}
Note that the equation in property (1) is satisfied by the map $f$ for some positive $k$, because each point in $I_1(f)$ has exactly two preorbits converging to $\Gamma$ and $I_1(f)$ is compact.
By the local stability of basic pieces it follows that local unstable manifolds of size $\delta$ for the map $g$ are $C^1$ close to those of $f$; moreover, as $k$ is fixed, the assertion holds also for $g$.\\
(2) As the strong transversality condition is open, then $I_1(g)$ is also compact. Assume by contradiction that there exist sequences $\{x_n\}$ and $\{y_n\}$ contained in $\Gamma(g)$ such that for each $n>0$ there exist points $z^1_n\neq z^2_n$, $d(z^1_n,z^2_n)<1/n$ and $z^i_n\in g^k(W^u_\delta(x_n,g))\cap g^k(W^u_\delta(y_n,g))$. Assuming that the sequences $\{z_n^i\}$ converge to a point $z$, a contradiction arises because $z\in I(g)$ turns to be a point of non transversal intersection between unstable manifolds.
\end{proof}

Let $B_0(A)$ denote the immediate basin of the attractor $A$.
To treat with the unstable intersections in $B_0=B_0(A)$, let
$$
L=L(f,p)=\cup_{\ell\geq 0}f^\ell(I_1(f)\cap B_\ell(f,p)),
$$
where $B_\ell=B_\ell(f,p)$ is defined inductively, as follows:
let $B_1=f^{-1}(B_0)\setminus B_0$ and for every $\ell>1$, $B_\ell=f^{-1}(B_{\ell-1})$.
\begin{lemma}
\label{lcompactL}
$L$ is compact and $f(L)\cap L=\emptyset$.
\end{lemma}
\begin{proof}
Note first that as $I$ is contained in $B$, then $I_1\cap \partial B_{\ell}=\emptyset$ for each $\ell\geq 0$. So $I_1\cap B_{\ell}$ is compact for every $\ell$. Moreover, by lemma \ref{closed} there exist at most finitely many values of $\ell$ for which the intersection of $I_1$ with $B_{\ell}$ is nonempty. This implies the first assertion.\\
Notice that $x\in M$ implies $\sharp (o^{-}(x)\cap I_1)\leq 1,$ where $o^-(x)$ denotes the union of all the preimages of $x$. If $x\in L$, then there exists $y\in I_1\cap B_\ell$ such that $f^{\ell}(y)=x$. It follows that $y$ is the unique point in $o^-(f(x))\cap I_1$, hence $f(x)\notin L$ by definition of $L$.
\end{proof}

\subsection{ Fundamental domains.}

Let $A$ be a periodic attractor such that the set of unstable intersections $I$ intersects the basin of $A$.
We can assume that $A$ is a fixed point and take $K$ equal to the closure of a fundamental domain such that $K$ is an annulus and
$L$ is contained in the interior of $K$. A fundamental domain with these properties can be constructed as in the proof of Theorem C in
\cite{ipr1}: there it was shown how to construct a fundamental domain containing in its interior a compact set $L$, provided $f(L)\cap L=\emptyset$.
Let $\partial_-K$ be the connected component of the boundary of $K$ that is closer to $p$, and $\partial_+K$ be the other component of the boundary of $K$ ($\partial_-K$ separates $p$ from $\partial_+K$ in $B_0$).
Let $Q$ be the component of $B_0\setminus K$ that contains $p$. It can also be assumed that for every $x\in Q\cap W^u(\Gamma)$ there exists $j>0$ and $y\in K\cap W^u(\Gamma)$ such that $f^j(y)=x$. This follows from the fact that the transversality condition implies that the preimage of $p$ cannot belong to $W^u(\Gamma)$, so it is possible to take a neighborhood $V_0$ of $p$ where just the points coming through $K$ can enter $V_0$; then one can take $K$ contained in $V_0$, and substitute $L$ by a homeomorphic image of it. Given $\delta >0$,  let $k>0$ be such that
\begin{equation}
\label{equK}
f^{-k}(K)\cap W^u(\Gamma ) \subset W^{u}_{\delta} (\Gamma).
\end{equation}

Observe that the number $k$ can be taken so as to satisfy also the conditions imposed to the number $k$ of lemma \ref{lemmaS}.
For each basic piece $\Lambda$ in $\Gamma$ let $\epsilon_s$ be a small positive number and define a fundamental domain for $\Lambda$:
$$
D(\Lambda)=D_{\epsilon_s}(\Lambda)=W^s_{\epsilon_s}(\Lambda)\setminus f(W^s_{\epsilon_s}(\Lambda)).
$$

\subsection{Local perturbations.}

It is classical that a perturbation can be performed as a finite sequence of perturbations with small supports.
Indeed, if $g$ is a perturbation of $f$ then there exists a diffeomorphism $t$, $C^1$ close to the identity, such that $g=ft$ (take $t(x)$ as the point in $f^{-1}(g(x))$ that is closest to $x$). Given a finite covering $\{W_1,\ldots, W_s\}$ of $M$, and a small perturbation $t$ of the identity, there exist $\{t_1,\ldots,t_s\}$ such that $t=t_1\ldots t_s$ and the support of $t_i$ is contained in $W_i$ (the support of $t_i$ is the closure of the set of points where $t_i(x)\neq x$).  For the proof of this, see \cite{ps}.

We begin taking an appropriate covering of $M$.
If $x\in W^u(\Lambda)$ for a basic piece $\Lambda\subset \Gamma$, then let $W$ be a neighborhood of $x$ compactly contained in the interior of
$$
\cup_{n\geq 0}f^n(V(\Lambda))\cup W^u(\Lambda),
$$
where $V(\Lambda)$ is a small neighborhood of $D(\Lambda)$. It is asked moreover that the closure of $W$ does not intersect $\Omega\setminus \Lambda$. If $x\in A$ then take $W$ compactly contained in the immediate basin of $A$ if $A$ is a nonperiodic attractor, and take $W$ compactly contained in the open set $Q$ defined in 5.2. If $x\notin W^u(\Gamma)\cup A$, then $W$ will be a neighborhood of $x$ such that the closure of $W$ does not intersect $W^u(\Gamma)\cup A$ (recall that the latter is closed by lemma \ref{unstable}).

From now on it is assumed that $g$ is a perturbation of $f$ and that the set of points where $f$ and $g$ are not equal is contained in an open set $W$ as the constructed above. On the other hand, it is clear that $t_i$ converges to the identity when $g$ converges to $f$.
We note that in any case, each attractor of $f$ has a fundamental domain $K$ and each basic piece has a fundamental domain $D(\Lambda )$ whose closures do not intersect that of $W$.

\subsection{Construction of unstable foliations.}

Let $\Lambda$ be a basic piece of $\Gamma$. We will follow the proof of de Melo \cite{dM}
A foliation $\mathcal F_f$ will be defined in a small neighborhood $V(\Lambda)$ of $D(\Lambda)$. Note that $D(\Lambda)$ has no intersections with the unstable set of $\Lambda$, but will certainly intersect unstable leaves of basic pieces that are greater to $\Lambda$ in the order of definition \ref{order}. Then an induction argument is applied: the foliation constructed in $V(\Lambda)$ must contain the leaves that are iterates of those unstable foliations constructed in previous steps and intersecting $D(\Lambda)$.\\
Once this foliation is defined in $V(\Lambda)$, the union of its forward iterates with $W^u_\delta(\Lambda)$ will contain a neighborhood $S_f$ of $\Lambda$. We will fix some positive number $\delta$ such that $S_f\supset W^u_\delta$. This procedure can be repeated with all the basic pieces in $\Gamma$, giving a foliation $\mathcal F_f$ defined in $S_f$.

By virtue of Lemma \ref{closed}, this neighborhood $S_f$ can be asked to satisfy the following properties:
\begin{enumerate}
\item
The restriction of $f$ to $S_f$ is injective.
\item
The intersection of $S_f$ with a basic piece of $f$ not contained in $\Gamma$ is empty.
\item
The intersection of $S_f$ with $I$ is empty.
\end{enumerate}

And as these properties are open, there exist a neighborhood $\mathcal U$ of $f$ such that for every $g\in\mathcal U$ there exists a neighborhood $S_g$ of $\Gamma(g)$ satisfying the properties enumerated above.

\begin{lemma}
\label{lfoliation}
If $g$ is a perturbation of $f$ that coincides with $f$ in a neighborhood of a fundamental domain of $\Gamma$, then there exists a foliation $\mathcal F_g$ defined in $S_g$ and an application $H_g: \mathcal F_f\to \mathcal F_g$ that satisfies the following properties.
\begin{enumerate}
\item Each leaf of $\mathcal F_g$ is an injectively immersed one dimensional manifold, that is contained in $W^u(x,g)$ whenever $x\in\Gamma(g)$ and is transverse to $W^s(\Gamma(g))$. Moreover, the tangent spaces to the leaves vary continuously.
\item For $x\in S_g$, denote by $F_x(g)$ the leaf through $x$.
If $y\in S_g$ and $g^n(x)=y$, then the connected component of $g^n(F_x(g))\cap S_g$ that contains $y$ is contained in $F_y(g)$.
\item
If $F_1$ is a leaf of $\mathcal F_f$, then $H_g(F_1)$ is a leaf of $\mathcal F_g$ that is close to $F_1$ in $C^0$ topology.
\item
$g(H_g(F_x(f)))\supset H_g(F_{f(x)}(f))$, whenever $x$ and $f(x)$ belong to $S_f$.
\end{enumerate}

\end{lemma}
\begin{proof}
Until now we have shown the existence of a foliation $\mathcal F_f$ of $S_f$ that satisfies (1) and (2) for the map $f$. Let $\Lambda_1(g)$ be a maximal basic piece and $V(\Lambda_1(g))$ a neighborhood of $D(\Lambda_1(g))$, where $f$ and $g$ coincide. Define $\mathcal F_g=\mathcal F_f$ in $V(\Lambda_1(g))$.
Then a neighborhood of $\Lambda_1(g)$ will be covered by the union of the future iterates of the leaves of $\mathcal F_g$ and $W^u_\delta(\Lambda_1(g))$.
After this, proceed by induction. Assume that $\mathcal F_g$ is defined in $\Lambda_i(g)$ for $1\leq i\leq j-1$ and satisfies (1) and (2). Then define $\mathcal F_g$ in a neighborhood $V(\Lambda_j(g))$ of $D(\Lambda_j(g))$ in such a way that it coincides with leaves that are forward iterates of leaves of $\mathcal F_g$ defined at previous stages and coincides with $\mathcal F_f$ elsewhere.
Again we cover a neighborhood of $\Lambda_j(g)$ taking the union of the future iterates of these foliations with $W^u_\delta(\Lambda_j(g))$. The continuity of the foliation follows in the same way as proved in \cite{dM}.

The application $H_g$ is defined as the identity for leaves of $\mathcal F_f$ contained in $V(\Lambda_1)$. Then, preserving conjugacy and continuity, extend $H_g$ to all the leaves of $\mathcal F_f$ that are contained in a neighborhood of $\Lambda_1$.  Proceed in the same way to cover neighborhoods of the remaining basic pieces.

Defined in this way, it is standard that $\mathcal F_g$  y $H_g$ satisfy the asserted properties.
\end{proof}
Let $\mathcal U$, $\rho$ and $k$ as in Lemma \ref{lemmaS}.

\begin{rk}
\label{rem}
a) If $g\in\mathcal U$, $F_x(g)$ and $F_y(g)$ are distinct leaves of $\mathcal F_g$ and $g^k(F_x(g))\cap g^k(F_y(g))\neq \emptyset$, then the intersection is transverse and the distance between two points in this intersection is at least $\rho/2$.\\
b) For every $x\in S_g$, the cardinality of $g^{-k}(g^k(x))\cap S_g$ is at most two.
\end{rk}

\begin{rk}
\label{importante}
Let $W$ be an open set as constructed at Subsection 5.3. The perturbation $g$ of $f$ coincides with $f$ outside $W$.
Once $W$ is known, one can choose fundamental domains $K$ of the attractors and $D(\Lambda)$ of the basic pieces in $\Gamma$ such that:
\begin{enumerate}
\item $W$ does not intersect $\cup_0^k \partial f^{-j}(K)$.
\item $W\cap V(\Lambda)=\emptyset$ for a neighborhood $V(\Lambda)$ of
the fundamental domain of the basic piece $\Lambda$ (this
validates the hypothesis of lemma \ref{lfoliation}).
\item
The set $W$ does not intersect neither the set $O:=\partial(\cup_{i}^{k} f^{i}(S_f))$, nor $\cup_0^k f^{-j}(O)$.
\end{enumerate}

From $f(W)=g(W)$ and item (3) above, it follows that
\begin{equation*}
\bigcup_{i=0}^{k}f^{i}(S_f)=\bigcup_{i=0}^{k}g^{i}(S_g).
\end{equation*}

\end{rk}

\subsection{Construction of the conjugacy on the basins of the attractors.}

The reader may keep in mind that the intersections of unstable manifolds occur at the basins of the attracting periodic orbits, never
on the basin of a nontrivial attractor (those having unstable manifold of dimension one).

Let $K$ be a fundamental domain of an attracting fixed point and assume that $L$ intersects this basin. The map $f$ is perturbed to a map
$g$ in an open set $W$ that satisfies the requirements of the above subsection.
The construction of the conjugacy begins in this fundamental domain. We can state now the fundamental step.

\begin{lemma}
\label{l0}
There exists a homeomorphism $h$ defined in $K$ such that:
\begin{enumerate}
\item
If $x$ belongs to $\partial_+K$, then $g(h(x))=h(f(x))$.
\item
If $x$ is a point in $S_f$ such that $f^j(x)\in K$ for some $j>0$, then $g^j(H(F_x(f)))$ contains $h(f^j(x))$.
\end{enumerate}
\end{lemma}

Before beginning with the proof of the lemma, we will construct a field of directions in a preimage of $K$ where the intersections of unstable leaves have not occurred yet.\\
Let $k>0$ be as defined in Lemma \ref{lemmaS} and equation \ref{equK} in subsection 5.2:
\begin{equation}
\label{equKa}
f^{-k}(K)\cap W^u(\Gamma)\subset W^{u}_{\delta}(\Gamma)\subset S_f.
\end{equation}

Note that $f^{-k}(K)$ does not contain future iterates of $I(f)$ and each component of $f^{-k}(K)$ is diffeomorphic to $K$. Denote by $K'$ the unique component of $f^{-k}(K)$ that is contained in $B_0$, and by $\partial_\pm K'$ the component of the boundary of $K'$ whose image by $f^k$ is equal to $\partial_\pm K$.

By Remark \ref{importante}, it is clear that $g^{-k}(K)=f^{-k}(K)$ and that the set
$$
S'=\{x\in g^{-k}(K)\cap S_g\ : \mbox{ there exists } y\in S_g
\mbox{ such that } y\neq x \mbox{ and }g^k(x)=g^k(y)\}
$$
does not depend on the perturbation made (although $f^k$ and $g^k$ do not necessarily coincide in $S'$).

\noindent  Given $g$ as above, we will now construct a continuous
field of directions $\chi_g$ defined in a neighborhood $V_1$ of
$g^{-k}(K)\cap S_g$. We will first construct $\chi_f$:

If $x\in S'$, then $\chi_f(x)$ is the unique direction such that
$Df_x^k(\chi_f(x))= Df_y^k(v_y)$, where $v_y$ is the direction
tangent to $F_y(f)$ at $y$, and $y$ as in the definition of $S'$.
Note first that the transversality condition and the construction
of $S_f$, imply that there exists at most one point $y\neq x$,
$y\in S_f$, such that $f^k(x)=f^k(y)$, so $\chi_f$ is well
defined.

Note that $\chi_f$ is continuous in $S'$ and that the transversality condition implies that
$\chi_f$ is transverse to the leaves of the foliation $\mathcal F_f$.

Now take points $x$ and $f(x)$, both contained in $f^{-k}(K)\cap S_f$ and define $\chi_f(x)$ transverse to
the leaf of $\mathcal F_f$ at $x$. Next define $\chi_f(f(x))=Df_x(\chi_f(x))$. Note that as $K$ is a fundamental domain, then
$x$ and $f(x)$ belong to the immediate basin of the attractor, so $x$ belongs to $\partial_+K'$. We have thus defined
$\chi_f$ at three closed disjoint sets: $S'\cap S_f$, $\partial_+K'\cap S_f$ and $f(\partial_+K')\cap S_f$. It is continuous
and transverse to the foliations.
To extend this field of directions to a neighborhood $V_1$ of $f^{-k}(K)\cap S_f$, one can use diverse techniques, for example
the averaging method employed in {\cite[Lemma 1.1]{dM}}.\\

\noindent
The objective now is to define a similar field of directions for $g$ close to $f$; we proceed as above, defining first
$\chi_g$ at a point $x$ in $S'$ as the preimage under $Dg^k_x$ of the image under $Dg^k_y$ of the direction of the leaf
$F_g(y)$, where $y$ is the (unique) point in $S_f$ such that $g^k(y)=g^k(x)$.
Recall that as was explained above, the set $S'$ does not depend on the perturbation. It is also known that $S_f$ does
not depend on the perturbation.
Moreover, as $L$ is interior to $K$, it follows that $S'$ is a neighborhood of $g^{-k}(L)$ contained in the interior of $g^{-k}(K)$.
It follows by the requirements made at Remark \ref{importante} that $\chi_f$ and $\chi_g$
coincide close to the boundary of $S'$. Thus $\chi_g$ can be extended to $V_1$ (where $\chi_f$ is defined), by imposing that
$\chi_g=\chi_f$ outside $S'$.\\

\noindent
For future reference, we state as a claim what we have had about the fields of directions.\\
{\bf Claim.}
There exists an open set $V_1$, neighborhood of $S_f\cap f^{-k}(K)$ such that, for each perturbation $g$ of $f$ that coincides with $f$ outside a set $W$ satisfying the conditions stated in Remark \ref{importante},
there exists a continuous field of directions $\chi_g$ satisfying the following properties:
\begin{enumerate}
\item
If $x$ belongs to $S'$ then $\chi_g(x)=Dg_x^{-k}(Dg_y^k(v_y))$, where $v_y$ is tangent to the leaf of $\mathcal F_g$ through $y$, and $y$ is as in the definition of $S'$.
\item
For every $x\in V_1$, $\chi_g(x)$ is transverse to $F_x(g)$.
\item
If $x$ and $g(x)$ belong to $V_1$, then $Dg_x(\chi_g(x))=\chi_g(g(x))$.
\item
$\chi_f=\chi_g$ outside $S'$.
\end{enumerate}

The integral line of $\chi_g$ through the point $x$ will be denoted by $\tilde \chi_g(x)$.
Note that the leaves of $\tilde\chi$ and the leaves of the foliation $\mathcal F_g$ give a
local product structure at $g^{-k}(K)\cap S_g$, precisely:\\
There exists a neighborhood $V_1$ of $g^{-k}(K)\cap S_g$ and a positive number $\delta_0$
such that, for every pair of points $x$ and $y$ in $g^{-k}(K)\cap S_g$ at a distance less
than $\delta_0$, the intersection of the leaf of $\tilde\chi_g$ through $x$ with the
leaf of $\mathcal F(g)$ through $y$ consists of exactly one point contained in $V_1$.
This intersection varies continuously with $x$ and $y$.\\

\noindent
{\bf{{\em Proof of lemma \ref{l0}.}}}

Recall that if $x\in K$, then, according to item (b) in Remark \ref{rem}, the cardinality of $f^{-k}(x) \cap S_f$ is equal to $0$, $1$ or $2$; define $h$ as follows:\\
1) If it is equal to zero, then $h(x)=x$.\\
2) If it is equal to one, let $y$ be the unique point in $f^{-k}(x)\cap S_f$. Let $y'$ be the unique point of intersection of the integral line of $\chi_g$ ($\chi_f(y)=\chi_g(y)$ because $y\notin S^{'}$) through $y$ with $ H_g (F_y(f))$. Then let $h(x)=g^k(y')$.\\
3) If it is equal to two, and $f^{-k}(x)\cap S_f=\{y_1,y_2\}$, then $h(x)$ is defined as the intersection of the leaves $g^k(H(F_{y_{1}}))$ and $g^k(H(F_{y_{2}}))$ with the ball of radius $\rho /4$ and centered at $x$.\\
To prove that $h$ is well defined, one must consider the third case. The intersection defining $h$ contains exactly one point. If the perturbation $g$ of $f$ is sufficiently small, then property (1) in Lemma \ref{lfoliation} implies that $F_y$ and $H(F_y)$ are close as $C^1$ embeddings, so their $g^k$-iterates are close as well. Moreover, by (a) in Remark \ref{rem}, the intersection defining $h$ is transverse and is unique at a distance less than $\rho/4$ from $x$.\\

Now we will prove the continuity of $h$. It depends on the location of the support of the perturbation. Assume first that $W$ is contained in the exterior of $\cup_{i\geq 0}f^i(S_f)$ (see Remark \ref{importante}). In this case $\mathcal F_f=\mathcal F_g$, $H_g$ is the identity, and consequently $h$ is equal to the identity. The same happens whenever $W$ intersects an attractor.

It remains to consider the case where the closure of $W$ is contained in the interior of $\cup_{i\geq 0}f^i(S_f)$.

Note that $h$ is continuous in the interior of $\cup_{i\geq 0}f^i(S_f)$ by the continuity of the foliations. On the other hand, $h$ is the identity in the complement of $\cup_0^kf^{i}(S_f)$, so it remains to prove the continuity in the boundary. Note that as $f(W)=g(W)$, then $f^k$ and $g^k$ coincide in
$$
\left (\cup_{n\geq 0} f^{n}(S_f) \setminus \cup_{n\geq 0} f^{n} (W)\right )  \cap K.
$$

Let $x$ be a point in the boundary of $(\cup_{i\geq 0}f^{i}(S_f))\cap K$, and $y$ a point in $f^{-k}(x)\cap S_f$. Then $y$ belongs to the boundary of $S_f$, so $F_y(f)=F_y(g)$ and $\chi_f(y)=\chi_g(y)$. Moreover, as $f^j(y)\notin W$ for every $0\leq j\leq k$, then there exists a neighborhood $U_y$ of $y$ such that $f^j(z)=g^j(z)$ for every $z\in U_y$. This implies that $h(x)=x$.

Note that $h$ is injective because if $x\in K$, then the cardinalities of $f^{-k}(x)\cap S_f$ and $g^{-k}(h(x))\cap S_f$ coincide since $f=g$ outside the future iterates of $S_f$.

To prove assertion (1) in the statement of the Lemma, take $x\in \partial_+K$. If $x$ does not belong to the union of the future iterates of $S_f$ then $h(x)=x$ and $h(f(x))=f(x)$, and (1) holds because $g$ and $f$ coincide at both points. In the remaining case, there exists a unique $y\in S_f$ such that $f^k(y)=x$ and we fall in case (2) of the definition of $h$. Moreover, $f(y)$ is the unique $f^k$-preimage of $f(x)$ contained in $S_f$. If $y'$ is the point of intersection of $\tilde\chi_g(y)=\tilde\chi_f(y)$ with $H_g(F_y(f))$, then by definition, $h(x)=g^k(y')$. By virtue of the third property of the fields of directions $\chi$, it comes that the $g(\tilde\chi_g(y))=\tilde\chi_g(g(y))$. In addition, $g(y')$ belongs to $g(H_g(F_y(f)))$ and therefore $g^k(g(y'))=h(f(x))$.

The second assertion of the lemma follows by construction.

\qed

This homeomorphism $h$ can be extended to the whole immediate basin of $p_f$ as follows: if $x\in B_0(p_f)$, then there exists a unique $j\in \Z$ such that $f_{|B_0(p_f)}^j(x)\in K$; then define $h(x)=g^{-j}_{|B_0(p_g)}(h(f^j(x)))$. This new extension of $h$ conjugates the restrictions of $f$ and $g$ to the corresponding immediate basins of the attractors. It is injective and open as $h$ is. Moreover, $h$ carries $p_f$ to $p_g$ and so its image contains a neighborhood of $p$: it follows that it is onto $B_0(p_g)$.\\

When $A$ is a nonperiodic attractor, the restriction of $f$ to a neighborhood of $A$ is injective. Moreover, as there cannot be unstable intersections in its basin, the arguments applied for diffeomorphisms allows to construct a local conjugacy $C^0$ close to the identity defined in the whole immediate basin of the attractor $A$ and satisfying the properties in the statement of Lemma \ref{l0}. Summing up, there exists a conjugacy $h$ from $B_0(f)$ (the union of the immediate basins of the attractors of $f$) onto $B_0(g)$.\\

Finally, we claim that for every point $x\in S_f$ and every nonnegative integer $j$ such that $f^j(x)\in B_0(f)$, it holds that $h(f^j(x))\in g^j(H(F_x(f)))$. Indeed, let $J(x)$ be minimum such that $f^J(x)\in B_0(f)$. Let $i\in \Z$ such that $f^{J+i}(x)\in K$. Note that $i$ is greater than or equal to zero, by the construction of $K$ (see subsection 5.2). Then the claim follows by the second assertion in the previous lemma.
\subsection{ Extension of $h$ to the whole manifold.}
For each point $x\in f^{-\ell}(B_0)$, one has precisely $d^{\ell}$ points in $g^{-\ell}hf^{\ell}(x)$ to choose $h(x)$. Our arguments will show that there exists one of these points closest to $x$.\\
Let $\Lambda_j\subset \Gamma$, $1\leq j\leq n$, $U_j$ a collection of disjoint neighborhoods and denote by $U'$ the union of the $U_j$. Let $U$ be a neighborhood of the attractors.\\ The proof of following lemma is straightforward.

\begin{lemma}
\label{l1}
There exists $N>0$ such that for each $x\in M$, the set
$\{j\geq 0\ :\ f^j(x)\notin U\cup U'\}$ contains at most $N$ elements.
\end{lemma}

Let $\alpha$ be the expansivity constant of the restriction of $f$ to $U\cup U'$ (recall that $f$ is a diffeomorphism from $U_j$ to $f(U_j)$ for $U_j\subset \Gamma$).
Let $\epsilon_0>0$ be such that $f(x)=f(y)$ implies $x=y$ or $d(x,y)> \epsilon_0$. Let $\epsilon<\min\{\alpha/2,\epsilon_0/2\}$.\\
If $U'$ is sufficiently small and $\mathcal F$ is the foliation obtained in Lemma \ref{lfoliation}, then the following additional property holds:\\
There exists constants $\delta>0$, $C>0$ and $\lambda>1$ such that $d(f^n(x)),f^n(y))\geq C\lambda ^nd(x,y)$ whenever $f^j(x)$ and $f^j(y)$ belong to the same leaf of $\mathcal F$ and $d(f^j(x),f^j(y))\leq \delta$ for every $0\leq j\leq n-1$. By changing the metric one can obtain $C=1$.\\
In what follows, $\mathcal U$ will be a neighborhood of $f$ such that the same properties hold for $g$ in $\mathcal U$.

\begin{lemma}
\label{l2}
Let $x$ and $y$ be points in $U'$ such that $f(y)=x$. If $h$ is defined in $x$ and $d(h(x),x)<\epsilon$, then $h$ can be defined in $y$ in such a way that $d(h(y),y)<\epsilon$.
\end{lemma}
\begin{proof} By the choice of $\epsilon_0$ and $\mathcal U$, it follows that $g^{-1}(h(x))\cap B(y;\epsilon_0/2)$ contains just one point, denoted $y'$. If $h(y)$ is defined as equal to $y'$, then we must prove that $d(y,y')<\epsilon$.\\
The proof of what follows is inspired in {\cite[Lemma 2]{ipr2}}. Assume first that $\Lambda_j\subset \Gamma$. Given a positive constant $\rho$ there exists a neighborhood $\mathcal U$ of $f$ such that the leaves $F_y(f)$ and $H(F_y(f)$ are $\rho$-close in $C^1$ topology. Then there exists a point $z\in H(F_y(f))$ such that $d(z,h(y))<\rho$. It follows that
\begin{eqnarray*}
d(g(h(y)),g(y))  & \geq & d(g(h(y)),g(z))-d(g(z),g(y))\\
& \geq & \lambda d(h(y),z)-K\rho\geq \lambda d(h(y),y)-\lambda\rho-K\rho,
\end{eqnarray*}
where $K$ is taken so that $d(g(z_1),g(z_2))\leq Kd(z_1,z_2)$ for every $g\in \mathcal U$ and $z_1$, $z_2$ in $M$.\\
On the other hand,
$$
d(g(h(y)),g(y))=d(h(f(y)),g(y))\leq d(h(x),x)+d(f(y),g(y))<\epsilon+\delta,
$$
where $\delta$ is the $C^0$ distance between $f$ and $g$. These equations imply that $d(h(y),y)\leq (\epsilon+\delta+\rho(\lambda+K))\lambda^{-1}$.This is less than $\epsilon$ if $\delta$ and $\rho$ are sufficiently small, which is obtained by diminishing $\mathcal U$.\\
It remains to consider the case where $\Lambda_j\subset E$, but the same argument (simplified because the basic piece is expanding) applies.
\end{proof}

Once a small $\epsilon>0$ and an integer $N$ are fixed, there exists $\epsilon'>0$ such that $d(h(f^j(x)),f^j(x))<\epsilon'$ for some $j\leq N$ implies $d(h(x),x)<\epsilon$.\\
On the other hand, one can make the restriction of $h$ to a neighborhood $U$ of the attractors as close to the identity as wished, say $d(h(y),y)<\epsilon'$ whenever $y\in U$. If a point $x$ belongs to $U'\cap B$, it is known by Lemma \ref{l1} that $f^k(x)\notin U'\cap U$ for at most $N$ iterates. Using Lemma \ref{l2} we conclude:

\begin{lemma}
\label{l3}
Given $\epsilon>0$ there exist a neighborhood $\mathcal U$ of $f$ and a number  $\epsilon'>0$ such that $d(h(y),y)<\epsilon'$ for every $y\in U$ implies that $d(h(x),x)<\epsilon$ for every $x\in U'\cap B$.
\end{lemma}

It follows that $d(h(x),x)<\epsilon$ for every $x\in B$: indeed, this follows by the previous lemma since a point $x\in B$ stays at most $N$ iterates in $B\setminus (U\cup U')$. It remains to prove that $h$ can be extended to the closure of $B$, that equals $M$. If $x\notin B$, then there exists $q$ such that $f^n(x)\in U'$ for every $n\geq q$. If $y=f^q(x)$ then there exists $j$ such that either $y\in \Lambda_j$ or $y\in W^s(\Lambda_j)\setminus \Lambda_j$ for some $j$.\\
Assume $y\in \Lambda_j$, and let $\{z_n\}$ be a sequence in $B$ convergent to $y$. Assume by contradiction that there exist two subsequences of $\{h(z_n)\}$ converging to different points. Say $h(z'_n)\to z$ and $h(z''_n)\to w$. Note that as $z$ and $w$ belong to $\Lambda_j(g)$, one can use the expansivity of $g$ in $\Lambda_j$ to assure that there exists an integer $m$ such that
\begin{equation}
\label{e1}
d(g^m(z),g^m(w))>2\epsilon.
\end{equation}
Once $m$ is fixed, we can choose $n$ large and points $f^m(z'_n)$ and $f^m(z''_n)$ that are arbitrarily close (if $m$ was negative, then $f^m(z'_n)$ is the preimage of $y$ that is closest to $f^m(y)\cap \Lambda_j$). As $h$ is $\epsilon$-close to the identity in $B$, it follows that
$$
d(h(f^m(z'_n)),h(f^m(z''_n)))<2\epsilon.
$$
But $h(f^m(z_n))=g^m(h(z_n))$, so taking $n$ large the above equation contradicts (\ref{e1}).\\
Consider now the case where $y\in W^s(\Lambda_j)\setminus\Lambda_j$. Assume by contradiction that there exist two sequences $z_n'$ and $z''_n$ in $B$ both converging to $y$ but such that $h(z'_n)$ and $h(z''_n)$ converge to different points $z$ and $w$. Clearly $z$ and $w$ belong to $H(F_y(f))$. It follows now that there exists a positive $m$ such that $d(g^m(z),g^m(w))>2\epsilon$. Reasoning as above provides a contradiction.

\section{More examples and Questions.}

The construction of the example in section 4 allows some generalizations. For instance, one can change the fixed attractor for a non periodic attractor. The result is a map that is $C^1$ $\Omega$-stable but all its perturbations are not $C^1$ stable, because the unstable intersections occur in the basin of an attractor that is not periodic, thus contradicting the Strong Transversality Condition.

Every covering map $f$ of the two-torus induces a linear map on the homology $H^1(T^2)=\R^2$. At the same time, a linear map $A$ in the plane  with integer coefficients, induces a covering of the two-torus whose induced map is $A$. Moreover, the map $f$ belongs to the same homotopy class of the linear map associated.

The example of section 4 induces the linear map $\left(%
\begin{array}{cc}
  2 & 0  \\
  0 & 1  \\
\end{array}%
\right)$.
It is easy to imitate the idea to produce a $C^1$ stable map whose induced map is the expanding matrix
$\left(\begin{array}{cc}
  2 & 0  \\
  0 & 2  \\
\end{array}%
\right)$.

In both cases, however, the covering is semi-conjugated to an expanding map. Indeed, Przytycki's example is semiconjugated to $z\mapsto z^2$ in $S^1$.

\noindent
{\bf Question 1:} Is every $C^1$ stable map in the two torus semi-conjugated to an expanding map?

An affirmative answer to this question would be very useful towards a classification of the $C^1$ stable maps of the two torus.

The homotopy class of a diffeomorphism always contains a $C^1$ stable diffeomorphism; also, the class of an expanding map contains
a $C^1$ stable map.

\noindent
{\bf Question 2:} Does every homotopy class contain a $C^1$ stable map?

If the answer to the second question is negative when the manifold is the two torus, then the answer to the first question is affirmative.

\end{document}